\documentclass[10pt,oneside]{amsart}
  
\usepackage{geometry} % See geometry.pdf to learn the layout options. There are lots.
\usepackage{amsmath, amsthm, amssymb}  
\usepackage{hyperref}   
\usepackage{lipsum} 
\usepackage{scrextend} 
\hypersetup{colorlinks=true,linkcolor=blue, linktocpage}
\usepackage{graphicx} % Use pdf, png, jpg, or eps with pdflatex; use eps in DVI mode
\usepackage{mathrsfs} % TeX will automatically convert eps --> pdf in pdflatex \usepackage[utf8]{inputenc}
\usepackage{color} 
\usepackage{verbatim}
\usepackage[bbgreekl]{mathbbol}
\usepackage[dvipsnames]{xcolor}
\usepackage{enumitem}
\usepackage{graphicx}

\usepackage{geometry}
\makeatletter
\@addtoreset{equation}{section}
\makeatother

\usepackage[all]{xy}
\usepackage{hyperref}
\title{A note on 3-subgroups in the space Cremona group}
\author{Konstantin Loginov}
\date{} % Activate to display a given date or no date

\newcounter{cthm}

\newtheorem{proposition}[equation]{Proposition}
\newtheorem{thm}[equation]{Theorem} 

\newtheorem{corollary}[equation]{Corollary}
\newtheorem{lem}[equation]{Lemma}
\theoremstyle{definition}

\newtheorem{setting}[equation]{Setting}
\newtheorem{question}[equation]{Question}

\newtheorem{exam}[equation]{Example}
\theoremstyle{exam}

\newcommand{\OOO}{\mathscr{O}}

\newcommand{\Addresses}{{% additional braces for segregating \footnotesize
  \bigskip
  \footnotesize
  
  \  
    \
   
  \textsc{Steklov Mathematical Institute of Russian Academy of Sciences, 8 Gubkina st., Moscow, Russia, 119991. }\\
  \textit{E-mail:} \texttt{loginov@mi-ras.ru}
}}
\begin{document}

\maketitle

\begin{abstract}
We prove that a finite $3$-group in the Cremona group $\mathrm{Cr}_3(\mathbb{C})$ can be generated by at most $4$ elements. This provides the last missing piece in bounding the ranks of finite $p$-subgroups in the space Cremona group.
\end{abstract}

\section{Introduction}
We work over the field $\mathbb{C}$ of complex numbers. The \emph{Cremona group} $\mathrm{Cr}_n(\mathbb{C})$ is the group of birational self-maps of the $n$-dimensional complex projective space $\mathbb{P}^n$. For $n=1$, this group is isomorphic to $\mathrm{PGL}(2, \mathbb{C})$. In contrast, for $n\geq 2$ the Cremona group becomes a very complicated object to work with. One way to understand its structure is by means of its finite subgroups. Such subgroups were classified for $n=2$ by I. Dolgachev and V. Iskovskikh, see \cite{DI09}. In dimension $3$, the situation is more delicate, and the complete classification seems to be out of reach. However, there exist classificational results for some classes of groups, see e.g. \cite{Pr09} in the case of simple groups. There are also numerous boundedness results for finite subgroups of the Cremona group, see \cite{PS16} and references therein.

We concentrate on one particular type of such subgroups, namely, the $p$-groups. By a \emph{$p$-group} we mean a finite group of order $p^k$ where $p$ is a prime number. A rank $r(G)$ of a $p$-group $G$ is defined as the minimal number of elements that generate $G$. The following problem is natural: estimate the rank of a $p$-subgroup in Cremona group $\mathrm{Cr}_n(\mathbb{C})$, cf. \cite{Ser09}. The complete answer to this question for $n=2$ was obtained by A. Beauville, see Theorem~\ref{thm-surfaces}. 

In higher dimensions, one can consider the group of birational self-maps of a rationally connected variety of a given dimension, and ask for a similar bound for the rank of its $p$-subgroups. For $n=3$, the sharp bounds were obtained for $p=2$ and $p\geq 5$, see Theorem \ref{thm-threefolds}. For $n=3$, $p=3$ in the work \cite{Kuz20} the bound $r(G)\leq 4$ was obtained modulo $3$ exceptional cases. In this work, we prove the following

\begin{thm}
\label{main-theorem}
Let $X$ be a $3$-dimensional rationally connected variety, and let $G\subset \mathrm{Bir}(X)$ be a finite $3$-group. Then $r(G)\leq 4$, and this bound is sharp. 
\end{thm}

\begin{corollary}
\label{main-corollary}
Let $G\subset \mathrm{Cr}_3 (\mathbb{C})$ be a finite $3$-group. Then $r(G)\leq 4$, and this bound is sharp. 
\end{corollary}

As a consequence, we can formulate the following theorem which provides sharp bounds for the ranks of $p$-subgroups acting on rationally connected threefolds.

\begin{thm}[{cf. \cite{Pr11}, \cite{Pr14}, \cite{PS17}, \cite{Kuz20}, \cite{Xu20}}]
\label{thm-threefolds}
Let $X$ be a rationally connected variety of dimension $3$, and let $G\subset \mathrm{Bir}(X)$ be a finite $p$-group. Then it can be generated by at most $r$ elements where 
\begin{itemize}
\item
$r=6$ when $p=2$,
\item
$r=4$ when $p=3$, 
\item
$r=3$ when $p\geq 5$. 
\end{itemize}
Moreover, these bounds are sharp.
\end{thm}

In the assumptions of Theorem \ref{thm-threefolds}, from \cite{Xu20} it follows that for $p\geq 5$ the group $G$ is abelian. For the discussion of analogous statements in the case $p=2, 3$ see Example \ref{exam-non-abelian} and Question \ref{question-abelian}. The claim that this bound is sharp follows from Example \ref{exam-2-sharp} for $p=2$, Example \ref{exam-3-sharp} for $p=3$ and the natural action of $(\mathbb{Z}/p)^3$ on $\mathbb{P}^3$ for $p\geq 5$. Let us formulate an analogous statement in dimension $2$.

\begin{thm}[{\cite{Beau07}}]
\label{thm-surfaces}
Let $X$ be a rational surface, and let $G\subset \mathrm{Bir}(X)$ be a finite abelian $p$-group. Then it can be generated by at most $r$ elements where 
\begin{itemize}
\item
$r=4$ when $p=2$, 
\item
$r=3$ when $p=3$, 
\item
$r=2$ when $p\geq 5$. 
\end{itemize}
Moreover, these bounds are sharp.
\end{thm}

The assumption that $G$ is abelian can be removed from Theorem \ref{thm-surfaces} due to Proposition \ref{reduction-to-abelian}. In fact, Beauville proves that if the bounds as above are attained then one can describe the conjugacy class of the given subgroup. It would be interesting to obtain analogous description in the $3$-dimensional case, cf. Question 1.9 in \cite{PS17}. Interesting results in this direction were obtained in \cite{Xu18}.

%Is it true that if the bound is attained, then $G$ is abelian (in which case Beauville describes such groups up to conjugation)? For $p\geq 5$ this is the case by \ref{Haution-Xu}, so what about $p=2, 3$?

\

We give a sketch of the proof of Theorem \ref{main-theorem}. We assume that a $3$-group $G$ of rank $\geq 5$ acts on a three-dimensional rationally connected variety by birational self-maps. We assume that the group $G$ is abelian: this can be achieved by passing to the quotient group by its Frattini subgroup, see Proposition \ref{reduction-to-abelian}. First we use a standard technique: resolve indeterminacy of the action of $G$ on $X$ and apply a $G$-MMP to obtain a $G\mathbb{Q}$-Mori fiber space with a faithful action of $G$. Then, using known results on $3$-groups in the plane Cremona group, we exclude the case of a positive-dimensional base. So we have to deal only with the case of a $3$-dimensional $G\mathbb{Q}$-Fano variety. In this case, we consider the action of $G$ on the anti-canonical linear system which turns out to be non-empty by orbifold Riemann-Roch theorem, see Section \ref{sect-anticanonical}. More precisely, there are two cases to examine: either the anti-canonical system contains one element up to multiplication by a scalar, or its dimension is at least $2$. 

In the first case, we study the geometry of the anti-canonical divisor, see Section \ref{sect-special-case}. We show that it is a reduced, irreducible surface that admits an action of a $3$-group of rank $r(G)-1$. Moreover, it is equivariantly birational either to a K3 surface, or to a rational surface, or to a ruled surface over an elliptic curve. In all these subcases we use estimates on the rank of $3$-groups that acts on such surfaces (see Section \ref{sect-action-on-K3}) to derive a contradiction.
 
In Section \ref{sect-general-case} we deal with the case when the anti-canonical system has dimension $\geq 2$. Since $G$ is abelian, we can find a $G$-invariant pencil of anti-canonical elements. If the singularities of this pencil are ``bad enough'', in Proposition \ref{prop-main-reduction} we prove that our variety is $G$-birational to a $G\mathbb{Q}$-Mori fiber space with the base of positive dimension, and we are done. Otherwise, we obtain a pencil of K3 surfaces with at worst canonical singularities. Again, using estimates on the rank of $3$-groups acting on such surfaces, we arrive at a contradiction. 

Our approach differs from that of \cite{Kuz20}: we do not use the classificational results on three-dimensional Gorenstein Fano varieties. This may be useful for possible generalisations.

\subsection*{Acknowledgements} 
The author thanks Yuri Prokhorov for reading the draft of the paper and many valuable comments, Alexandra Kuznetsova and Constantin Shramov for useful discussions. This work is supported by Russian Science Foundation under grant 18-11-00121.

\section{Preliminaries}
We work over the field of complex numbers. All the varieties are projective and defined over $\mathbb{C}$ unless  stated otherwise. We use the language of the minimal model program (the MMP for short), see e.g. \cite{KM98}, and the MMP for linear systems, cf. \cite{Al94}. 

\subsection{Contractions} By a \emph{contraction} we mean a projective morphism $f\colon X \to Y$ of normal varieties such that $f_*\OOO_X = \OOO_Y$. In particular, $f$ is surjective and has connected fibers. A \emph{fibration} is defined as a contraction $f\colon X\to Y$ such that $\dim Y<\dim X$. 

Let $G$ be a finite group. By a \emph{$G$-variety} we mean a variety $X$ together with an action of a group $G$. If $f$ is a $G$-equivariant contraction (resp., a $G$-equivariant fibration) of $G$-varieties, we call it a \emph{$G$-contraction} (resp., a \emph{$G$-fibration}).

\subsection{Pairs and singularities} A \emph{pair} $(X, B)$ consists of a normal variety $X$ and a boundary $\mathbb{Q}$-divisor $B$ with coefficients in $[0, 1]$ such that $K_X + B$ is $\mathbb{Q}$-Cartier. {A \emph{pair} $(X, B)$ \emph{over $Z$} is a pair $(X, B)$ together with a projective morphism $X\to Z$}. Let $\phi\colon W \to X$ be a log resolution of $(X,B)$ and let $K_W +B_W = \phi^*(K_X +B)$.
The \emph{log discrepancy} of a prime divisor $D$ on $W$ with respect to $(X, B)$ is $1 - \mu_D B_W$ and it is denoted by $a(D, X, B)$. We say $(X, B)$ is lc (resp. klt) (resp. $\epsilon$-lc) if $a(D, X, B)$ is $\geq 0$ (resp. $> 0$)(resp. $\geq \epsilon$) for every $D$. Note that if $(X, B)$ is $\epsilon$-lc, 
then automatically $\epsilon \leq 1$ because $a(D, X, B) = 1$ for almost all $D$.

A \emph{lc-place} of $(X, B)$ is a prime divisor $D$ over $X$, that is, on birational models of $X$, such that $a(D, X, B) = 0$. A \emph{lc-center} is the image on $X$ of a lc-place.

We say that a group $G$ acts on the pair $(X, B)$ if $X$ is a $G$-variety and the boundary divisor $D$ is $G$-invariant.
%\emph{Sub-pairs} and their singularities are defined similarly by letting the coefficients of $B$ to be any real number $\leq 1$. In this case we similarly have the notions of sub-lc, sub-klt, etc.

%\subsection{Minimal model program (MMP)}
%We will use standard results of the minimal model program (cf. \cite{KM98}). Assume $(X, B)$ is a pair over $Z$. Assume $H$ is an ample$/Z$ $\mathbb{Q}$-divisor and that $K_X + B + H$ is nef$/Z$. Suppose $(X, B)$ is klt or that it is $\mathbb{Q}$-factorial dlt. Then we can run an MMP$/Z$ on $K_X + B$ with scaling of $H$. If $(X, B)$ is klt and if either $K_X + B$ or $B$ is big$/Z$, then the MMP terminates \cite{BCHM10}. In this paper we mainly work in dimension 3 in which case termination is known in full generality for lc pairs.

\subsection{Log Fano and log Calabi-Yau pairs} 
\label{sect-log-CY}
Let $(X, B)$ be an lc pair over $Z$. We say $(X,B)$ is \emph{log Fano over $Z$} if $-(K_X+B)$ is ample over $Z$. If $Z$ is a point then $(X, B)$ is called a \emph{log Fano pair}. In this case if $B=0$ then $X$ is called a \emph{Fano variety}. 
We say $(X,B)$ is \emph{log Calabi-Yau over $Z$} if $K_X + B \sim_{\mathbb{Q}} 0$ over $Z$. If $Z$ is a point then $(X, B)$ is called a \emph{log Calabi-Yau pair}. 

\subsection{Mori fiber space}
A $G\mathbb{Q}$-\emph{Mori fiber space} is a $G\mathbb{Q}$-factorial variety $X$ with at worst terminal singularities together with a $G$-contraction $f: X\to Z$ to a variety $Z$ such that $\rho^G(X/Z)=1$ and 
$-K_X$ is ample over $Z$. If $X$ is $G$-factorial (e.g. if $X$ is smooth), we call it $G$-\emph{Mori fiber space}.

\

%\subsection{Terminal singularities} Do we need any results on terminal singularities?

We formulate results of O. Haution that are applicable to the study of $p$-subgroups in the Cremona group. Recall that a \emph{Chern number} of a smooth projective variety is the intersection number of Chern classes of its tangent bundle.
 
\begin{thm}[{\cite[1.1.2]{Ha19}}]
\label{chern-divisible}
Let $G$ be a $p$-group, and let $X$ be a smooth projective variety endowed with an action of $G$. Assume that either $G$ is abelian, or $\dim X < p$. If $G$ has no fixed points on $X$, then every Chern number of $X$ is divisible by $p$.
\end{thm}

\begin{corollary}
\label{divisible-by-p}
Let $p>2$ be a prime number, and let $G$ be a abelian $p$-group. Suppose that $X$ is a terminal $G\mathbb{Q}$-factorial threefold whose non-Gorenstein singularities are of the type $\frac{1}{2}(1,1,1)$. Moreover, assume that $G$ has no fixed points on $X$. Then the integer $2^3(-K_X)$ is divisible by $p$. 
\end{corollary}
\begin{proof}
Since non-Gorenstein singular points on $X$ are of type $\frac{1}{2}(1,1,1)$, we have $(-K_X)^3 = \frac{a}{2}$ where $a\in \mathbb{Z}$. Let $f: X' \to X$ be a $G$-equivariant resolution \cite{AW97}. We may assume that $f$ is an isomorphism outside finite number of singular points of $X$. Write
\[
K_{X'} = f^* K_X + \sum a_i E_i.
\]
Since $2$ is invertible modulo $p$, we can consider the following formula modulo $p$: 
\[
\left(-K_{X'}\right)^3 = \left( - f^* K_X - \sum a_i E_i \right)^3 = (- K_X)^3 - \sum_Q a_Q
\]
where $a_Q$ is a contribution made by exceptional divisors over a singular point $Q\in X$. By Theorem \ref{chern-divisible}, the integer $(-K_{X'})^3$ is divisible by $p$. Since $G$ has no fixed points on $X$, the exceptional divisors $E_i$ form $G$-orbits whose length is divisible by $p$, so it follows that $\sum_Q a_Q$ is equal to $0$ modulo $p$. Note that $2^3(-K_X)^3$ is an integer since $-2K_X$ is a Cartier divisor. Hence $2^3(-K_X)^3$ is divisible by $p$, and the claim follows.
\end{proof}

\begin{thm}[{\cite[1.2.1]{Ha19}}]
Let $G$ be a $p$-group, and let $X$ be a smooth projective variety with an action of $G$. Assume that either $G$ is cyclic, or $\dim X < p - 1$. Then $G$ has no fixed points on $X$ if and only if the Euler characteristic $\chi(X,F)$ of every $G$-equivariant coherent $\OOO_X$-module $F$ is divisible by~$p$.
\end{thm}

\begin{corollary}[{\cite[Main theorem]{Xu20}}]
\label{Haution-Xu}
Let $X$ be a rationally connected variety of dimension n. If $G \subset \mathrm{Bir}(X)$ is a 
$p$-subgroup and $p > n + 1$, then $G$ is abelian and the rank of $G$ is at most $n$. 
\end{corollary}

As a consequence, if $n=3$ then any $p$-group $G\subset \mathrm{Bir}(X)$ for $p\geq 5$ is abelian and has rank $\leq 3$. Now, we present examples of the actions of $p$-group.

\begin{exam} 
\label{exam-2-sharp}
Consider a faithful action of the group $(\mathbb{Z}/2)^2$ on $\mathbb{P}^1$ given by the matrices 
\begin{equation*}
\begin{pmatrix}
0 & 1 \\
1 & 0 
\end{pmatrix},\ \ \ \ \ 
\begin{pmatrix}
-1 & 0 \\
0 & 1 
\end{pmatrix},
\end{equation*}
Then the group $G=(\mathbb{Z}/2)^6$ faithfully acts on $\mathbb{P}^1\times\mathbb{P}^1\times \mathbb{P}^1$. This shows that the bound $r(G)\leq 6$ in Theorem \ref{thm-threefolds} is sharp for $p=2$.
\end{exam}

\begin{exam} 
\label{exam-3-sharp}
Let $S$ be a Fermat cubic surface defined by the equation 
\[
x_0^3 + x_1^3 + x_2^3 + x_3^3 = 0
\] 
in $\mathbb{P}^3$. Let the group $(\mathbb{Z}/3)^3$ act on $S$ by multiplication by a cube root of unity on of the coordinates $x_i$, $1\leq i\leq 3$. Then the group $G=(\mathbb{Z}/3)^3 \times \mathbb{Z}/3$ faithfully acts on $X = S\times \mathbb{P}^1$ where the action of a second factor $\mathbb{Z}/3$ on the projective line is standard. This shows that a $3$-group $G$ of rank $4$ can be embedded in $\mathrm{Cr}_3(\mathbb{C})$, so the bound $r(G)\leq 4$ in Theorem \ref{thm-threefolds} is sharp for $p=3$. %Similarly, one can define an action of a $3$-group of rank $4$ on a Fermat cubic threefold.
\end{exam}

\begin{exam}
\label{exam-non-abelian}
For $p=3$ there exists an action of a non-abelian group on a rational threefold. Indeed, consider an action of the group $G = \mathcal{H}_3\times \mathbb{Z}/3$ of rank $3$ on $X = \mathbb{P}^2\times \mathbb{P}^1$. Here $\mathcal{H}_3=(\mathbb{Z}/3\times\mathbb{Z}/3)\rtimes \mathbb{Z}/3$ is the Heisenberg group of rank $2$ and order $27$, cf. \cite[A.2]{Kuz20}. Similarly, one can construct an example of a non-abelian $2$-group of rank $5$ acting on a rational threefold, cf. \cite[Example 2.2]{PS17}.
\end{exam}

The following question seems natural.

\begin{question}
\label{question-abelian}
Assume that for $p=2$ (resp., $p=3$) a $p$-group $G$ with $r(G)=6$ (resp., $r(G)=4$) faithfully acts on a rationally connected $3$-dimensional variety. Is it true that $G$ is abelian? 
\end{question}

%It seems that this is the case: one can show that if $r=4$ then $|G|=81$. It follows from the fact that in both cases, that is, of a Fano fibration and of a fibration into K3 surfaces, we can write fibration exact sequences for $G$ and obtain the result.

The next proposition shows that to bound the rank of $p$-groups acting on a rationally connected variety it is enough to bound the rank of abelian $p$-groups.

\begin{proposition}[{\cite[3.1]{Kuz20}}]
\label{reduction-to-abelian}
Assume that for all rationally connected varieties $Y$ of dimension $n$ and any abelian $p$-subgroup $A \subset \mathrm{Bir}(Y)$ we have $r(A)\leq N$ for some number $N$. If $X$ is rationally connected and of dimension $n$, then for any $p$-subgroup $G \subset \mathrm{Bir}(X)$ (in particular, non-abelian) we have $r(G)\leq N$.
\end{proposition}

We include the idea of its proof for the reader's convenience. Recall that a \emph{Frattini subgroup} $\Phi(G)$ of a $p$-group $G$ is an intersection of all maximal subgroup of $G$. It is normal and has the following remarkable property: the quotient group $G/\Phi(G)$ is isomorphic to an elementary abelian $p$-group $(\mathbb{Z}/p)^{r(G)}$. Passing to a resolution of indeterminacy of the group action, we may assume that a $p$-group $G$ faithfully acts on a rationally connected variety $X$ of dimension $n$. Then the quotient group $G/\Phi(G)$ of rank $r(G)$ faithfully acts on a rationally connected variety $X/\Phi(G)$, and the claim follows.

The next proposition bounds the rank of a $3$-group that acts on a three-dimensional terminal singularity.

\begin{proposition}[{\cite[2.4]{Pr11}, \cite[4.4]{Kuz20}}]
\label{invariant-point-bound}
Assume that $X$ is a threefold with terminal singularities and $G$ is a $3$-subgroup in $\mathrm{Aut}(X)$ fixing a point $x \in X$. Then $r(G)\leq 3$.
\end{proposition}

\section{$G$-action on K3 surfaces}
\label{sect-action-on-K3}
In this section, we consider actions of abelian $p$-groups on (possibly, singular) K3 surfaces and their generalisations. By a K3 surface we mean a normal surface $S$ with at worst canonical (that is, $1$-lc) singularities such that $\mathrm{H}^1(\OOO_S)=0$ and $K_S\sim 0$. 

\begin{proposition}
\label{canonical-K3}
For $p\geq 3$, assume that a $p$-group $G$ faithfully acts on a K3 surface $S$. Then $r(G)\leq 2$.
\end{proposition}
\begin{proof}
Passing to the minimal resolution we may assume that $S$ is a smooth K3 surface with a faithful action of the group $G$. Then the claim follows from the work of V. Nikulin \cite[4.5]{Nik80}. 
\iffalse
We provide some details. There exists an exact sequence
\[
1\to \text{Aut}_s (S) \to \text{Aut} (S) \to \mathbb{Z}/I \to 1
\]
where $\text{Aut}_s (S)$ is the subgroup of symplectic automorphisms, that is, automorphisms that preserve a holomorphic $2$-horm on $S$, and $I$ is some integer. Put $G_s= G\cap \mathrm{Aut}_s(S)$. Then $|G_s|\leq 8$ (cf. \cite[15.1.8]{Hu16}), hence $r(G_s)\leq 1$. But since the image of $G$ in $\mathbb{Z}/I$ is cyclic, the group $G$ has rank $\leq 2$.
\fi
\end{proof}

%Note that by classification of groups of order $p^2$ either $G\simeq (\mathbb{Z}/p\mathbb{Z})^2$ or $G\simeq \mathbb{Z}/p^2\mathbb{Z}$. Hence if $r(G)=2$ then we get the first possibility.

Now we consider actions of $p$-groups on log Calabi-Yau surface pairs, see \ref{sect-log-CY} for the definition. Such pairs generalize the notion of K3 surfaces. Let $(S, \Delta)$ be a lc log Calabi-Yau surface pair endowed with an action of a $p$-group $G$. In particular, the surface $S$ is normal and the boundary $\mathbb{Q}$-divisor $\Delta$ is $G$-invariant.

\begin{proposition}
\label{non-canonical-logCY}
For $p\geq 3$, assume that a $p$-group acts on a log Calabi-Yau surface pair $(S, \Delta)$ which is lc but not $1$-lc (in particular, this is satisfied either if $\Delta\neq 0$ or $\Delta=0$ and $S$ is not canonical). Then $r(G)\leq 3$.
\end{proposition}
\begin{proof}
Using the assumption that $(S, \Delta)$ is not $1$-lc and taking the minimal resolution of $S$, we may run the $G$-MMP to obtain a $G$-Mori fiber space $S'\to Z$ where $S'$ is a smooth surface, and $Z$ is either a point or a curve. 

If $Z$ is a point then $S'$ is isomorphic to $\mathbb{P}^2$, but a $p$-group of rank $> 3$ cannot act on $\mathbb{P}^2$ faithfully, for example, by Theorem \ref{thm-surfaces}. Hence we may assume that $Z$ is a smooth curve. It is well known that in this case $Z$ is either a rational curve or a curve of genus $1$, see e.g. \cite[4.1]{BL20}. If $Z$ is a rational curve then $S'$ is rational, so again $r(G)\leq 3$ by Theorem \ref{thm-surfaces}.

So we assume that $Z$ is a smooth curve of genus $1$. Note that $Z$ admits a faithful action of a $p$-group $G'$ of rank $r(G)-1$. It is well known that, after fixing the base point on $Z$, we have $\mathrm{Aut}(Z)\simeq Z\rtimes \mathbb{Z}/k$ where $Z$ is identified with the subgroup of translations, and $k$ can be equal to $2,4,6$. For $p\geq 5$, clearly $G\subset Z$, and hence $G'$ sits in the $p$-torsion subgroup of $Z$ which is isomorphic to $(\mathbb{Z}/p)^2$. Therefore in this case $r(G')\leq 2$ and $r(G)\leq 3$. For $p=3$, we have to pay attention to the case $k=3$ and $k=6$. In this case, there is a subgroup isomorphic to $(\mathbb{Z}/3\times \mathbb{Z}/3)\rtimes \mathbb{Z}/3$ in $\mathrm{Aut}(Z)$. However, one checks that the rank of such subgroup is equal to $2$, hence $r(G')\leq 2$. Therefore $r(G)=3$, and the proposition is proven. 
%The only such curve is an elliptic curve with complex multiplication whose endomorphism ring is $\mathbb{Z}[\xi_3]$ where $\xi^3=1$. But the only $3$-group of rank $3$ acting on such $Z$ faithfully is non-abelian. The proof is complete. 
\end{proof}

\section{Anticanonical linear system}
\label{sect-anticanonical}
In this section, we estimate the dimension of the anti-canonical linear system $\mathrm{H}^0(X, \OOO_X(-K_X))$ of a $G\mathbb{Q}$-Fano variety $X$ that admits an action of an abelian $3$-group of rank $\geq 5$. The assumption that $G$ is abelian is justified by Proposition \ref{reduction-to-abelian}.

\subsection{Orbifold Riemann-Roch}
By \cite[10.2]{Re85}, for a terminal threefold $X$ and a Weil $\mathbb{Q}$-Cartier divisor $D$ on it we have the following version of the Riemann-Roch formula:
\begin{equation}
\label{eq-ORR}
\chi (\OOO_X(D)) = \chi (\OOO_X) + \frac{1}{12}D(D-K_X)(2D-K_X) + \frac{1}{12}D\cdot c_2(X) + \sum_Q c_Q (D)
\end{equation}
where for any cyclic quotient singularity $Q$ we have 
\begin{equation}
c_Q(D) = -i \frac{r^2-1}{12r} + \sum_{j=1}^{i-1}\frac{\overline{bj} (r-\overline{bj}) }{2r}.
\end{equation}
Here $r$ is the index of $Q$, the divisor $D$ has type $i \frac{1}{r}(a, -a, 1)$ at $Q$, $b$ satisfies $ab=1 \ \mathrm{mod}\ r$, and $\overline{\ }$ denotes the residue modulo $r$. Non-cyclic non-Gorenstein singularities correspond to a basket of cyclic points. Moreover, by \cite[10.3]{Re85} one has 
\begin{equation}
\label{eq-Miyaoka}
(-K_X)\cdot c_2(X) + \sum_Q (r-1/r) = 24
\end{equation}

The following proposition shows that the number of non-Gorenstein points is bounded by $9$ under our assumptions. We omit its proof which follows from Proposition \ref{invariant-point-bound} and formulas \eqref{eq-ORR}, \eqref{eq-Miyaoka}.

\begin{proposition}[{\cite[4.9]{Kuz20}}]
\label{bound-non-Gorenstein-points}
Assume that a $3$-group $G$ of rank $r$ acts on a $G\mathbb{Q}$-factorial terminal Fano threefold $X$, and that $X$ is non-Gorenstein. Then $r\leq 5$ and if $r=5$ then non-Gorenstein points on $X$ have type $\frac{1}{2}(1,1,1)$ and form one orbit of length $9$.
\end{proposition}

From now on we assume that $r(G)=5$. In the next proposition we compute the dimension of the anti-canonical system on $X$ under this assumption.

\begin{proposition}
\label{rank-5-dichotomy}
Let $X$ be a $G\mathbb{Q}$-Fano threefold where $G$ is a $3$-group of rank $5$. Then either
\begin{itemize}
\item
$(-K_X)^3=1/2$ and $\dim |-K_X| = 0$, or
\item
$\dim |-K_X|\geq 1$. 
\end{itemize}
\end{proposition}
\begin{proof}
Consider two cases: when $X$ is Gorenstein and when it is not. In the first case, we may use the usual Riemann-Roch formula \eqref{eq-ORR} without the correction terms $c_Q$. From Kawamata-Viehweg vanishing theorem it follows that $\chi(\OOO_X(-K_X))=h^0(\OOO_X(-K_X))$ and $\chi(\OOO_X) = 1$, so for $D=-K_X$ we obtain
\[
4 \leq g+2 = h^0 (\OOO_X(-K_X)) = 3 + \frac{1}{2} \left(-K_X\right)^3,
\]
where the number $g$ is called the \emph{genus} of $X$. The claim is proven in this case.

Now assume that $X$ is non-Gorenstein. For $D = -K_X$ and singular points of type $\frac{1}{2}(1,1,1)$ we have $c_Q = -1/8$, so if there are $N$ such points, in the formula \eqref{eq-ORR} they give a contribution $-N/8$. From \eqref{eq-Miyaoka} it follows that 
\begin{equation}
(-K_X)\cdot c_2(X) = 24 - \frac{3N}{2}.
\end{equation}
We have
\begin{equation}
\label{eq-rr-half-points}
%\chi (\OOO_X(-K_X)) &= 1 + \frac{1}{2}(-K_X)^3 + \frac{1}{12} ( 24 - 3N/2 ) - 3N/24, \\
h^0 ( \OOO_X (-K_X) ) = 3 + \frac{1}{2}\left(-K_X\right)^3 - \frac{N}{4} 
\end{equation}
Since by Proposition \ref{bound-non-Gorenstein-points} we have $N=9$, using $(-K_X)^3\geq 1/2$ we obtain 
\begin{equation}
h^0 ( \OOO_X (-K_X) ) = \frac{3}{4} + \frac{1}{2}\left(-K_X\right)^3 \geq 1
\end{equation}
We see that $h^0(\OOO_X (-K_X)) = 1$ if and only if $(-K_X)^3=1/2$, and if $(-K_X)^3\geq 1$ then $h^0(\OOO_X (-K_X))\geq~2$.
\end{proof}

%By \ref{divisible-by-p} we have that $(-K_X)^3 = a/2$ where $a$ is divisible by $p$. In particular, $\mathrm{H}^0(\OOO_X(-K_X))$ has a $G$-invariant linear subsystem $\mathscr{H}$ of dimension $2$. Indeed, this can be seen as follows: pick two $G$-invariant elements $D_1$, $D_2$ in $|-K_X|$, and take their span.

\section{$G$-action on threefolds. Special case}
\label{sect-special-case}
In this section, we work in the following setting.
\begin{setting}
\label{setting-special-case}
Let $X$ be a terminal $G\mathbb{Q}$-factorial Fano threefold where $G$ is an abelian $3$-group of rank $r = 5$. Assume that $(-K_X)^3=1/2$ and $\dim |-K_X| = 0$. Denote the unique anti-canonical element in $|-K_X|$ by $S$. 
\end{setting}

This case can be excluded by Corollary \ref{divisible-by-p}. However, we give an alternative proof.

\iffalse
\begin{lem}
\label{lem-positive-integer}
Let $X$ be a $3$-dimensional variety with isolated singularities. Let $L$ be an ample Cartier divisor, and let $S$ be an effective Weil divisor. Then the intersection number $L\cdot L \cdot S$ is a positive integer. 
\end{lem}
\begin{proof}
\end{proof}
\fi
\begin{proposition}[{\cite[2.9]{Pr11}}]
\label{the-pair-is-lc}
The pair $(X, S)$ is lc.
\end{proposition}
\begin{proof}[Sketch of proof]
If the pair $(X, S)$ is not lc then pick a maximal $\lambda<1$ such that the pair $(X, \lambda S)$ is lc. Note that it is an lc log Fano pair. Consider its minimal lc-center $Z$. After perturbing the pair by a $G$-invariant divisor with small coefficients we may assume that $Z$ is a unique, and hence $G$-invariant, lc-center of the lc log Fano pair $(X, \Delta)$. Moreover, $Z$ is either a point or a smooth rational curve. Considering an action of $G$ on $Z$ and on the normal bundle to the generic point on $Z$ and using Proposition \ref{invariant-point-bound}, we conclude that in both cases we have $r(G)\leq 3$ which contradicts our assumptions.
\end{proof} 

\begin{proposition}
\label{S-reduced-irreducible}
$S$ is irreducible and reduced. 
\end{proposition}
\begin{proof}
By our assumptions $S^3=1/2$ and $2S$ is Cartier. Assume that $S$ is either reducible or non-reduced. Write $S =\sum k_i S_i$ in $\mathrm{Cl}(X)$ for some $k_i\geq 1$. %Note that $S$ is Cohen-Macaulay \cite[5.25]{KM98}. 
%Then if $S$ is reducible, there exist two components of $S$, say $S_i$ and $S_j$, which intersect in a curve. 
The intersection number $(2S)^2\cdot D$ for a Cartier divisor $2S$ and a Weil divisor $D$ is an integer, see e.g. \cite[1.34]{KM98}. We have $(2S)^2\cdot D>0$ by Nakai-Moishezon criterion, see \cite[1.37]{KM98}. Write
\[
2 = \sum k_i S_i \cdot (2S)^2 \geq \sum k_i
\]
and conclude that either $S=S_1+S_2$ or $S=2S_1$. Since $G$ is a $3$-group, in both cases the components of $S$ are $G$-invariant. 

Since $\rho^G(X) = 1$, it follows that the components of $S$ are proportional to $S$ in $\mathrm{Cl}(X)$ modulo the torsion subgroup. By Proposition \ref{the-pair-is-lc} we may assume that the pair $(X, S)$ is lc. Then the normalization $S_1^\nu$ of $S_1$ is a lc del Pezzo surface endowed with an action of a $3$-group $G'$ of rank $r(G) - 1 = 4$. Indeed, this can be seen by considering the action of $G$ on $S_1$ and on the tangent space to its generic point. If $S_1^\nu$ is klt then it is rational and we derive a contradiction to Theorem \ref{thm-surfaces}. If $S_1^\nu$ is strictly lc then we may apply the same argument as in Proposition \ref{the-pair-is-lc} for $\lambda=1$ to derive a contradiction. This shows that $S$ is irreducible and reduced. 
\end{proof}

\begin{proposition}
\label{special-case-does-not-occur}
The case $r(G) = 5$ as in the setting \ref{setting-special-case} does not occur.
\end{proposition}
\begin{proof}
By Proposition \ref{S-reduced-irreducible} we may assume that $S$ is reduced and irreducible. Considering the action of $G$ on the tangent space to the generic point of $S$, we see that a $p$-group $G'$ of rank $4$ faithfully acts on $S$. Let $\nu: S'\to S$ be the normalization of $S$. Consider the pair $(S', \Delta')$ defined by the formula
\[
K_{S'} + \Delta' = \nu^* K_S \sim 0.
\]
Here $\Delta'$ is an effective Weil $\mathbb{Q}$-divisor called the \emph{different}, see \cite{Ka07}. Moreover, $\nu(\Delta')$ is supported on the non-normal locus of $S$, the pair $(S', \Delta')$ is lc, and $\Delta'$ is $G'$-invariant. If $\Delta' = 0$ and $S'=S$ is canonical, we apply Proposition \ref{canonical-K3}. Indeed, we have $K_S \sim 0$ by adjunction, and from the exact sequence
\[
0 \to \OOO_X(K_X) \to \OOO_X \to \OOO_S \to 0
\]
it follows that $h^1(\OOO_S) = 0$, so $S$ is a K3 surface with at worst canonical singularities. By Proposition \ref{canonical-K3} we conclude that $r(G')\leq 2$, hence $r(G)\leq 3$. If either $\Delta'\neq 0$ or $\Delta' = 0$ and $S'=S$ is not canonical then we apply Proposition \ref{non-canonical-logCY} and obtain $r(G')\leq 3$, so $r(G)\leq 4$. This contradicts to the assumption $r(G)=5$, and the proposition if proven. 
\end{proof}

\section{$G$-action on threefolds. General case}
\label{sect-general-case}
In this section, we work in the following setting. Let $X$ be a terminal $G\mathbb{Q}$-factorial Fano threefold where $G$ is a $3$-group of rank $r(G)=5$, and $(-K_X)^3>1/2$. 

Using Proposition \ref{reduction-to-abelian}, we will assume that $G$ is abelian. Then by Proposition \ref{rank-5-dichotomy} we have $\dim |-K_X|\geq 1$, and since a representation of $G$ in $\mathrm{H}^0(X, \OOO(-K_X))$ splits a direct sum of $1$-dimensional representations, we see that $|-K_X|$ has a $G$-invariant linear subsystem $\mathscr{H}$ of dimension $2$. We will need the following simple lemma. 

\begin{lem}
\label{lem-action-on-MFS}
Assume that a $p$-group of rank $(\mathbb{Z}/p)^r$ for $p>2$ faithfully acts on 
%a $G\mathbb{Q}$-Fano threefold $X$ which is $G$-birationally equivalent to a 
a $3$-dimensional rationally connected $G\mathbb{Q}$-Mori fiber space $X\to Z$ with the base of positive dimension. Then $r(G)\leq 3 + \delta_{p, 3}$.
\end{lem}
\begin{proof}
%By assumption, $X$ is $G$-birational to some $G\mathbb{Q}$-Mori fiber space $X'\to Z$ with $\dim Z>0$. 
Consider an exact sequence
\[
1\to G_1 \to G \to G_2 \to1
\]
where $G_2$ acts faithfully on the base $Z$, and $G_1$ acts faithfully acts on the generic fiber of $X\to Z$. If $\dim Z = 1$ then $Z \simeq \mathbb{P}^1$, and $G_2\subset \mathbb{Z}/p$, so $G_1$ acts faithfully acts on a smooth geometrically rational surface and hence has rank $\leq 2 + \delta_{p,3}$ according to Theorem \ref{thm-surfaces}, so the claim follows. If $\dim Z = 2$, we argue analogously.
\end{proof}

The next proposition is the crucial technical tool in this paper.

\begin{proposition}[cf. {\cite[3.1]{Al94}, \cite[6.6]{Pr09}}]
\label{prop-main-reduction}
Assume that $X$ is a normal terminal threefold endowed with a faithful action of a finite group $G$, and let $\mathscr{H}$ be a non-empty $G$-invariant pencil (that is, a linear system of dimension $1$) on $X$ without fixed components, such that
\[
-K_X - \mathscr{H} \sim E
\]
where $E$ is an effective $\mathbb{Q}$-divisor. Assume that either
\begin{itemize}
\item
$E=0$, and the pair $(X, \mathscr{H})$ is not canonical, or
\item
$E \neq 0$,
\end{itemize}
Then $X$ is $G$-birationally equivalent to a $G\mathbb{Q}$-Mori fiber space with the base of positive dimension.
\end{proposition}
\begin{proof}
We start with the first case. Assume that the linear system $\mathscr{H}$ does not have fixed components, and that the pair $(X, \mathscr{H})$ is not canonical. Let $c$ be its canonical threshold, that is, $c$ is maximal with the property that $(X, c \mathscr{H})$ is canonical. Clearly, $0<c<1$. Let $f: (X', c \mathscr{H}')\to(X, c \mathscr{H})$ be a $G$-equivariant $G\mathbb{Q}$-factorial terminalization. We have 
\[
K_{X'}+c \mathscr{H}' = f^*(K_X+c \mathscr{H}),
\]
\[ 
K_{X'}+\mathscr{H}' + E = f^*(K_X+\mathscr{H})\sim 0,
\]
where $E$ is a non-zero effective integral (since $\mathscr{H}\sim -K_X$) $f$-exceptional $\mathbb{Q}$-divisor. Note that $E$ may be reducible. Since $-K_X-c\mathscr{H}$ is ample, we have that $-K_{X'}-c\mathscr{H}'$ is big and nef. We run a $G$-equivariant  $(K_{X'}+c \mathscr{H}')$-MMP and obtain a $G\mathbb{Q}$-Mori fiber space $(\overline{X}, \overline{\mathscr{H}})\to Z$. 

Note that $\overline{X}$ is terminal. Indeed, we show that each step of $(X', c\mathscr{H}')$-MMP is also a step of the usual $K_{X'}$-MMP. This follows from the fact that on the terminalization $X'$ the linear system $\overline{\mathscr{H}}$ does not contain curves in its base locus \cite[1.22]{Al94}, and hence it is nef. On each step, the pair remains terminal, so its base locus does not contain curves, and hence the linear system remains nef, so the same argument applies on each step of the MMP.

We may assume that $Z$ is a point, otherwise the proposition is proven. So $\rho^G(\overline{X})=1$, and we have 
\[
K_{\overline{X}} + \overline{\mathscr{H}} + \overline{E} \sim 0, 
\]

From \cite[3.1]{Al94} it follows that $\overline{E}\neq 0$. Then 
\[
- K - \overline{\mathscr{H}} \sim \overline{E},
\]
and we have reduced the first case of the lemma to the second. 

\

Now we consider the second case, so we assume that $E\neq 0$. Pass to a $G\mathbb{Q}$-factorial terminalization $f: (X', \mathscr{H}')\to (X, \mathscr{H})$. Write 
\[
K_{X'}+\mathscr{H}'+E' = f^* ( K_X + \mathscr{H} + E ) \sim 0
\]
where $E' \neq 0$. %We may assume that $E'$ is a boundary divisor since as we have shown above the pair $(X, \mathscr{H})$ may be assumed to be lc???. 
Apply a $G$-MMP to the pair $(X', \mathscr{H}')$. On each step of this MMP we contract a $(K_{X'} + \mathscr{H}')$-negative extremal ray. But $-(K_{X'} + \mathscr{H}') \sim E'$, hence $E'$ cannot be contracted by the negativity lemma. Hence the $G$-MMP terminates with a $G\mathbb{Q}$-Mori fiber space $\overline{X}\to Z$. If the dimension of $Z$ is positive, the proposition is proven. So we assume that $Z$ is a point, $\rho^G(\overline{X})=1$, and 
\[
K_{\overline{X}} + \overline{\mathscr{H}} + \overline{E} \sim 0.
\]
Write $-K_{\overline{X}}\sim a \overline{\mathscr{H}}$ for some $a>1$. As explained above, the pair $(\overline{X}, \overline{\mathscr{H}})$ is terminal. Hence by \cite[1.22]{Al94} the general member $\overline{H} \in \overline{\mathscr{H}}$ is smooth and contained in the smooth locus of $\overline{X}$. By adjunction formula we obtain 
\[
K_{\overline{H}} = ( K_{\overline{X}} + \overline{H} ) |_{\overline{H}} = ( a - 1 ) \overline{H}|_{\overline{H}},
\]
hence $\overline{H}$ is a smooth del Pezzo surface. Note that the linear system $\overline{\mathscr{H}}$ defines a rational map $\overline{X}\dashrightarrow \mathbb{P}^1$. Resolve its base locus of $\overline{\mathscr{H}}$ and run a relative $G$-MMP over $\mathbb{P}^1$ to obtain a Mori fiber space with the base of positive dimension.
\end{proof}

%$K_{\overline{X}} + \overline{\mathscr{H}}\sim 0$. We have $\kappa(K_{\overline{X}} + c\overline{\mathscr{H}})=-\infty$, but $K_{\overline{X}} + c\overline{\mathscr{H}}\sim -(1-c) \overline{\mathscr{H}}$

\begin{corollary}
\label{cor-canonical}
Let $X$ be a $3$-dimensional $G\mathbb{Q}$-Fano variety that admits a $G$-invariant anti-canonical pencil, that is, a $G$-invariant linear system $\mathscr{H}$ such that
$\dim \mathscr{H}= 1$, and 
$\mathscr{H}\subset |-K_X|$. 
Then either
\begin{itemize}
\item
$X$ is $G$-birationally equivalent to a $G\mathbb{Q}$-Mori fiber space with the base of positive dimension, or
\item
the pair $(X, \mathscr{H})$ is canonical, the linear system $\mathscr{H}$ does not have fixed components, and 
$X$ is $G$-birationally equivalent to a $G$-equivariant fibration over $\mathbb{P}^1$ whose general fiber is a smooth K3 surface.
\end{itemize} 
\end{corollary}
\begin{proof}
By Proposition \ref{prop-main-reduction}, we may assume that the pair $(X, \mathscr{H})$ is canonical and does not have fixed components. 
%We prove that in this case a general element in $\mathscr{H}$ is a normal irreducible K3 surface with at worst canonical singularities. First we show that a general element in $\mathscr{H}$ is irreducible. Assume that a general element $D\in \mathscr{H}$ is singular in codimension $1$. Then since $X$ is smooth in codimension $2$, this contradicts to the fact that the pair $(X, \mathscr{H})$ is canonical. This shows that $D$ is smooth in codimension $1$. Since $\mathscr{H}\sim -K_X$ is ample, the support of $D$ is connected, and hence $D$ is irreducible. Finally, from \cite[1.21]{Al94} it follows that $D$ is normal and has at worst canonical singularities.
 
We resolve indeterminacy of the rational map $X \dashrightarrow \mathbb{P}^1$ given by $\mathscr{H}$. Consider a terminalization $(X', \mathscr{H}')\to(X, \mathscr{H})$. Then the linear system $\mathscr{H}'$ is basepoint free.%, and $X'$ is Gorenstein.
We have
\[
K_{X'} + \mathscr{H}' = f^* ( K_X + \mathscr{H} ) \sim 0,
\] 
and the pair $(X', \mathscr{H}')$ is terminal. From \cite[1.23]{Al94} it follows that $\mathscr{H}'$ may have only isolated basepoints (that are non-singular on $X'$) such that the multiplicity of a g.e. in $\mathscr{H}'$ at them equals $1$. We claim that $\mathscr{H}'$ cannot have isolated basepoints. Indeed, since $\mathscr{H}'$ is generated by two elements that  are strict preimages of $D_1$, $D_2$ on $X$, the base locus of $\mathscr{H}'$ coincides with their intersection. But since the base locus lives in the non-singular locus of $X'$, two surfaces cannot intersect in a point there, so the base locus if $\mathscr{H}'$ is empty. Hence the general element of $\mathscr{H}'$ is a smooth K3 surface, and the corollary is proven.
\end{proof}
%To prove the second assertion, note the following. We may assume that $f$ is obtained by first blowing-up the $9$ points of type $\frac{1}{2}(1,1,1)$ and then taking the terminalization. Clearly, 

Hence we obtain a $G$-morphism $X'\to \mathbb{P}^1$. But then a group of rank $r(G)-1=4$ faithfully acts on a general fiber which is a smooth K3 surface. This contradicts to Proposition \ref{canonical-K3}.

Now, we summarise the obtained results to prove Theorem \ref{main-theorem}. 

\section{Proof of the main results}

\begin{proof}[Proof of Theorem \ref{main-theorem}]
Assume that $G$ is a $3$-group of rank $\geq 5$ that faithfully acts on a rationally connected $3$-dimensional variety $Y$. By Proposition \ref{reduction-to-abelian} we may assume that $G$ is abelian. Passing to a $G$-equivariant resolution of singularities and running a $G$-MMP, we may assume that $G$ faithfully acts on a $G\mathbb{Q}$-Mori fiber space $X\to Z$. From Lemma \ref{lem-action-on-MFS} we deduce that $Z$ is a point, so $X$ is a $3$-dimensional $G\mathbb{Q}$-factorial Fano variety with $\rho^G(X)=1$.

By Proposition \ref{rank-5-dichotomy} either $(-K_X)^3=1/2$ and $\dim |-K_X| = 0$, or $\dim |-K_X|\geq 1$. We show that the first case does not occur. For a unique element $S\in |-K_X|$, by Proposition \ref{the-pair-is-lc} we may assume that the pair $(X, S)$ is lc. Then, by Proposition \ref{S-reduced-irreducible} the divisor $S$ is reduced and irreducible. Proposition \ref{special-case-does-not-occur} shows that $S$ is a K3 surface with at worst canonical singularities, but then a $3$-group of rank $r(G)-1\geq 4$ cannot faithfully act on $S$ by Proposition \ref{non-canonical-logCY}.

So we see that $\dim |-K_X|\geq 1$. Hence there exists a $G$-invariant pencil $\mathscr{H}$, that is, a $G$-invariant subsystem of dimension $1$, of anti-canonical elements. We consider the pair $(X, \mathscr{H})$. %By Proposition \ref{the-pair-is-lc} we may assume that this pair is lc. Indeed, it is spanned by two $G$-invariant elements, say $D_1$ and $D_2$, so we may assume that both pair $(X, D_1)$ and $(X, D_2)$ are lc, but then the pair $(X, \mathscr{H})$ is lc as well. 
Applying Proposition \ref{prop-main-reduction} and Lemma \ref{lem-action-on-MFS} we deduce that $\mathscr{H}$ does not have base components and that the pair $(X, \mathscr{H})$ is canonical. But then by Corollary \ref{cor-canonical} we have that $X$ is $G$-birational to a fibration over $\mathbb{P}^1$ whose general fiber is a smooth K3 surface. A group of rank $r(G)-1=4$ faithfully acts on the general fiber which contradicts to Proposition \ref{canonical-K3}. 

This proves that $r(G)\leq 4$. The fact that this bound is sharp follows from Example \ref{exam-3-sharp}.
\end{proof}

\begin{proof}[Proof of Corollary \ref{main-corollary}]
Follows from Theorem \ref{main-theorem} and the fact that the variety $X$ in Example \ref{exam-3-sharp} is rational.
\end{proof}

\

\Addresses

\end{document}